\documentclass[12pt]{article}

\usepackage{amsmath}
\usepackage{amsfonts}
\usepackage{color}
\usepackage{amssymb}
\usepackage{graphicx,bm}
\usepackage{enumerate}
\usepackage{amsthm,amscd}
\usepackage[all]{xy}

\def\tr{{\rm tr}}

 \allowdisplaybreaks

\usepackage{hyperref}

\newtheorem{theorem}{Theorem}[section]

\newtheorem{lemma}[theorem]{Lemma}
\newtheorem{problem}[theorem]{Problem}

\begin{document}

\title{New proofs of stability theorems on spectral graph problems\thanks{This work was supported by  NSFC (Grant No. 11931002).  
 E-mail addresses: ytli0921@hnu.edu. cn (Y\v{o}ngt\={a}o L\v{i}),  
ypeng1@hnu.edu.cn (Yu\`{e}ji\`{a}n P\'{e}ng, corresponding author).}}

\author{
Yongtao Li, Yuejian Peng$^{\dag}$ \\[2ex]
{\small School of Mathematics, Hunan University} \\
{\small Changsha, Hunan, 410082, P.R. China }  }

\maketitle

\vspace{-0.5cm}

\begin{abstract}
Both the Simonovits stability theorem 
and the Nikiforov spectral stability theorem are   powerful tools 
for solving exact values of  Tur\'{a}n numbers   
in extremal graph theory. 
Recently, F\"{u}redi [J. Combin. Theory Ser. B 115 (2015)] provided
 a concise and contemporary proof  of the Simonovits stability theorem. 
In this note, we  present a unified treatment for some extremal graph problems, 
including 
short proofs of  Nikiforov's spectral stability theorem and 
the clique  stability theorem 
proved recently by  Ma and Qiu [European J. Combin. 84 (2020)]. 
Moreover, some  spectral extremal problems related to 
the $p$-spectral radius and signless Laplacian radius are also 
included. 
 \end{abstract}



\section{Introduction}

Extremal graph theory is one of the significant branches of 
discrete mathematics and has experienced 
an impressive growth during the last few decades. 
It  deals usually with the problem of determining or 
estimating the maximum or minimum possible
size of a graph  which satisfies certain requirements, 
and further characterize the extremal graphs attaining the bound. 
Such problems are related to other areas including 
theoretical computer science, 
discrete geometry, information theory and number theory.

We say that a graph $G$ is $F$-free if it does not contain 
  an isomorphic copy of $F$ as a subgraph. 
For given $n$, 
 the {\em Tur\'an number} of a graph $F$, denoted by $\mathrm{ex}(n,F)$, 
  is the maximum number of edges 
  in an $n$-vertex $F$-free graph.   
  An $F$-free graph on $n$ vertices with $\mathrm{ex}(n, F)$ edges is called an {\em extremal graph} for $F$. 
  
 Let $K_{r+1}$ be the complete graph on $r+1$ vertices. 
    In 1941, Tur\'{a}n \cite{Turan41} solved the  natural question of determining
 $\mathrm{ex}(n,K_{r+1})$ for $r\ge 2$. 
 Let $T_r(n)$ denote the complete $r$-partite graph on $n$ vertices where 
 its part sizes are as equal as possible. 
Tur\'{a}n \cite{Turan41}   extended a result of Mantel  
and  obtained that if $G$ is a $K_{r+1}$-free graph on $n$ vertices, 
then $e(G)\le e(T_r(n))$, equality holds if and only if $G=T_r(n)$.   
There are many extensions and generalizations on Tur\'{a}n's result. 
 In the language of extremal number, the Tur\'{a}n theorem can be stated as 
\begin{equation} \label{eqTuran}
\mathrm{ex}(n,K_{r+1}) = e(T_r(n)). 
\end{equation}  
Moreover, we can easily see that $(1-\frac{1}{r}) \frac{n^2}{2} - \frac{r}{8}\le e(T_r(n)) \le (1-\frac{1}{r}) \frac{n^2}{2} $. 
  It is  a cornerstone of extremal graph theory 
  to 
  understand $\mathrm{ex}(n, F)$  for various graphs $F$; 
  see \cite{FS13, Keevash11, Sim13} for surveys. 
The problem of determining $\mathrm{ex}(n, F)$ is usually called the 
Tur\'{a}n-type extremal problem. 
 The most celebrated extension of Tur\'{a}n's theorem always attributes to a result of 
 Erd\H{o}s, Stone and Simonovits, 
 although it was proved first in \cite{ES66}, but indeed easily follows from a result of Erd\H{o}s and Stone \cite{ES46}.

\begin{theorem}[Erd\H{o}s--Stone--Simonovits, 1946/1966]  \label{thm11}
If $F$ is a graph with chromatic number 
$\chi (F)=r+1$, then 
\[  \mathrm{ex}(n,F) =e(T_r(n)) + o(n^2)= \left(  1-\frac{1}{r} + o(1)\right) \frac{n^2}{2}. \]  
\end{theorem}

The Tur\'{a}n theorem implies that 
every $n$-vertex graph with more than $(1-\frac{1}{r})\frac{n^2}{2} $ 
edges contains a copy of $K_{r+1}$.  
The Erd\H{o}s--Stone--Simonovits theorem states that 
for any integer $t$ and  $\varepsilon >0$, 
then for sufficiently large $n$, 
every $n$-vertex graph with at least 
$(1-\frac{1}{r})\frac{n^2}{2}+\varepsilon n^2$ edges not only 
contains a copy of $K_{r+1}$, but also contains a copy of $K_{r+1}(t)$, 
the complete $(r+1)$-partite graph with $t$ vertices in each part.  

In 1966, 
Erd\H{o}s \cite{Erd1966Sta1,Erd1966Sta2} and Simonovits \cite{Sim1966} proved 
a stronger structural  theorem of Theorem \ref{thm11} and 
discovered that 
this extremal problem exhibits a certain stability phenomenon. 
Let $G_1$ and $G_2$ be two graphs which are defined on the same vertex set. 
The {\it edit-distance} between $G_1$ and $G_2$, 
denoted by $d(G_1,G_2)$, is the minimum integer $k$ 
such that $G_1$ can be obtained from $G_2$ 
by adding or deleting a total number of $k$ edges. 
The following structural stability theorem was  proved 
by Erd\H{o}s \cite{Erd1966Sta1,Erd1966Sta2} and Simonovits in \cite{Sim1966}. 
This result  bounds the edit-distance 
between $G$ and $T_r(n)$ when $G$ is $F$-free and $e(G)$  
is close  to $(1-o(1))\mathrm{ex}(n,F)$.

\begin{theorem}[Erd\H{o}s--Simonovits, 1966] \label{thm12}
Let $F$ be a graph with $\chi (F)=r+1\ge 3$. 
For every $\varepsilon >0$, 
there exist $\delta >0$ and $n_0$ such that 
if $G$ is a graph on $n\ge n_0$ vertices, 
 and $G$ is $F$-free such that 
$e(G)\ge (1- \frac{1}{r} - \delta) \frac{n^2}{2}$, 
 then  the edit distance $d(G,T_r(n)) \le \varepsilon n^2$. 
\end{theorem}

Roughly speaking,  
if $G$ is an $n$-vertex 
 $K_{r+1}$-free graph for which $e(G)$ is close to $e(T_r(n))$, 
then the structure of $G$ must resemble the Tur\'{a}n graph  
in an appropriate sense.  
Over the past twenty years, the  stability theorem has attracted wide public concern and 
plays an important role in the 
development of extremal  graph theory.

\subsection{Spectral extremal graph problems}

Let $G$ be a simple graph on $n$ vertices. 
The \emph{adjacency matrix} of $G$ is defined as 
$A(G)=[a_{ij}]_{n \times n}$ where $a_{ij}=1$ if two vertices $v_i$ and $v_j$ are adjacent in $G$, and $a_{ij}=0$ otherwise.   
We say that $G$ has eigenvalues $\lambda_1 ,\ldots ,\lambda_n$ if these values are eigenvalues of 
the adjacency matrix $A(G)$. 
Let $\lambda (G)$ be the maximum  value in absolute 
 among the eigenvalues of $G$, which is 
 known as the {\it spectral radius} of graph $G$, 
 that is, 
 \[  \lambda (G) = \max \{|\lambda | : \text{$\lambda$ 
 is an eigenvalue of $G$}\}. \]
By the Perron--Frobenius theorem, 
the spectral radius of a graph $G$ is actually 
the largest eigenvalue of $G$ 
since the adjacency matrix $A(G)$ is nonnegative.  
The spectral radius of a graph sometimes can give some information  
about the structure of graphs.  

In  the classic extremal graph problem,  
we usually study the maximum or minimum 
number of edges that the extremal graphs can have. 
Correspondingly, 
we can study the extremal spectral problem.  
We define $\mathrm{ex}_{\lambda}(n,F)$ 
to be the largest eigenvalue of the adjacency matrix 
in an $F$-free $n$-vertex  graph, that is, 
\[ \mathrm{ex}_{\lambda}(n,F):=\max \bigl\{ \lambda(G): |G|=n~\text{and}~F\nsubseteq G \bigr\}. \]

In 2007, Nikiforov \cite{Niki2007laa2} showed the 
spectral version of the Tur\'{a}n theorem. 
\begin{equation} \label{eqeq2}
\mathrm{ex}_{\lambda}(n,K_{r+1}) 
=\lambda (T_r(n)).
\end{equation}

By calculation, we can obtain that $(1-\frac{1}{r})n - \frac{r}{4n} \le \lambda (T_r(n))\le 
(1-\frac{1}{r})n$. 
It should be mentioned that 
the spectral version of the  Tur\'{a}n theorem was 
early studied independently by Guiduli in his PH.D. dissertation \cite[pp. 58--61]{Gui1996} 
dating back to 1996 under the guidance of L\'{a}szl\'{o} Babai. 
We remark here that the proof of Guiduli  for the spectral 
Tur\'{a}n theorem  is completely different from that of Nikiforov \cite{Niki2007laa2}. 
The main idea in his proof   \cite{Gui1996} reduces the problem of bounding the largest 
spectral radius among $K_{r+1}$-free graphs 
to complete $r$-partite graphs, then one can show further that 
the balanced complete $r$-partite graph attains the maximum 
value of the 
spectral radius. 
The proof of Nikiforov is more algebraic and relies on a profound theorem from \cite{Niki2002cpc} as well as  
an old theorem from \cite{Zykov1949} and \cite{Erd1962}.

\medskip 

A natural question we may ask is the following: 
what is the relation between the  spectral Tur\'{a}n theorem and the edge Tur\'{a}n theorem? 
Does the spectral  bound imply the 
edge bound of Tur\'{a}n's theorem?  
This question was also proposed in \cite{Niki2009jctb}. 
The answer is positive.  
It is well-known that 
${e(G)} \le \frac{n}{2}\lambda (G)$, 
with equality if and only if  $G$ is  regular. 
Although the Tur\'{a}n graph $T_r(n)$
is sometimes not regular, but it is nearly regular. 
Upon calculation, we can verify that 
$   {e(T_r(n))} = \left\lfloor \frac{n}{2}\lambda (T_r(n)) \right\rfloor$.   
With the help of this observation, 
the spectral Tur\'{a}n theorem implies that  
\begin{equation} \label{eqeqq3}
 e(G) \le 
\left\lfloor \frac{n}{2} \lambda (G) \right\rfloor 
\le \left\lfloor \frac{n}{2} \lambda (T_r(n)) \right\rfloor =e(T_r(n)).
\end{equation}  
Thus the spectral Tur\'{a}n theorem implies the 
classical Tur\'{a}n theorem. 

\medskip

In 2009, Nikiforov  \cite{Niki2009cpc} 
proved the following theorem, 
 which determined the asymptotic maximum spectral radius of 
 $F$-free graphs for arbitrary graph $F$. 
Theorem \ref{thm13} is a spectral analogue of the Erd\H{o}s--Stone--Simonovits Theorem \ref{thm11}. 

\begin{theorem}[Nikiforov, 2009] \label{thm13}
If $F$ is a graph with chromatic number $\chi (F)=r+1$, 
then 
\[  \mathrm{ex}_{\lambda} (n,F) =\lambda (T_r(n)) + o(n)=
\left( 1-\frac{1}{r} + o(1)\right) n . \] 
\end{theorem}

In the same year, Nikiforov \cite{Niki2009jgt} proved the corresponding spectral 
analogue of the Erd\H{o}s--Simonovits stability theorem.

\begin{theorem}[Nikiforov, 2009] \label{thm14}
Let $F$ be a graph with $\chi (F)=r+1\ge 3$. For every $\varepsilon >0$,  
there exist $\delta >0$ and $n_0$ such that 
if $G$ is an $F$-free graph on $n\ge n_0$ vertices  
 and  
$\lambda(G)\ge (1- \frac{1}{r} - \delta) n$, 
 then  the edit distance $d(G,T_r(n)) \le \varepsilon n^2$. 
\end{theorem}

Since the Rayleigh formula gives $2e(G)/n \le \lambda (G)$, 
 the spectral theorem of Nikiforov 
is a generalization of  the Erd\H{o}s--Stone--Simonovits theorem.  
Moreover, the spectral stability theorem  also generalizes  the 
Simonovits stability theorem.  
Just like the significance of the classical Erd\H{o}s--Simonovits stability theorem, 
with the development of spectral extremal graph theory, 
we  believe commonly that the spectral stability theorem 
will also play a vital role in solving the spectral extremal problems;  
see, e.g., \cite{CFTZ20,CDT2022,LP2021,DKLNTW2021} 
for recent progress.

\subsection{Generalized extremal graph problems}

Let $k_s(G)$ denote the number of copies of $K_s$ in $G$. 
In particular, we have $k_1(G)=v(G)$ and $k_2(G)=e(G)$. 
In 1949, Zykov \cite{Zykov1949}, and Erd\H{o}s \cite{Erd1962}  independently 
proved an extension of 
the Tur\'{a}n theorem, which states that 
if $G$ is an $n$-vertex $K_{r+1}$-free graph, 
then  $ k_s (G)\le k_s (T_r(n))$ for every $s=2,3,\ldots ,r$, 
equality holds if and only if $G$ is the Tur\'{a}n graph 
$T_r(n)$. 
For two graphs $T$ and $F$, 
the {\it generalized Tur\'{a}n number }
 $\mathrm{ex}(n,T,F)$ is defined as 
the maximum number of copies of $T$ in an $F$-free graph on $n$ vertices.  
For example, setting $F=K_2$, the extremal number  
$\mathrm{ex}(n,K_2,F)$ is the classical function $\mathrm{ex}(n,F)$.  
Under this definition, the result of Zykov and Erd\H{o}s can be written as 
\begin{equation} \label{eqeq4}
\mathrm{ex}(n,K_s, K_{r+1}) = k_s(T_r(n)). 
\end{equation}

In 2016, 
Alon and Shikhelman \cite{AS2016} 
systematically studied the function $\mathrm{ex}(n, T , F)$ for many 
various combinations of
$T$ and $F$. In particular, they proved  
the following theorem.

\begin{theorem}[Alon--Shikhelman, 2016] \label{thm15}
If $F$ is a graph with chromatic number $\chi (F)=r+1$, 
then for every $2\le s \le r$, we have 
\[  \mathrm{ex}(n,K_s, F) = k_s (T_r(n)) + o(n^s)
=(1+o(1)) {r \choose s}\left( \frac{n}{r}\right)^s . \] 
\end{theorem}

The following clique version of the stability theorem 
was proved by Ma and Qiu \cite[Theorem 1.4]{MQ2020}. 
 Obviously, taking the case $s=2$, 
we can see that Theorems \ref{thm15} and \ref{thm16} 
reduce to Theorems \ref{thm11} and \ref{thm12} respectively.  

\begin{theorem}[Ma--Qiu, 2020] \label{thm16}
Let $F$ be a graph with $\chi (F)=r+1\ge 3$. 
For every $\varepsilon >0$, 
there exist $\delta >0$ and $n_0$ such that 
if $G$ is an $F$-free graph on $n\ge n_0$ vertices 
 and 
$k_s(G)\ge {r \choose s} (\frac{n}{r})^s -\delta n^s$ for some $2\le s \le r$, 
 then  the edit distance $d(G,T_r(n)) \le \varepsilon n^2$. 
\end{theorem}

 The original proofs of both the  classical stability theorem of Erd\H{o}s and Simonovits,
  the spectral stability theorem of Nikiforov and the clique stability theorem 
  of Ma and Qiu 
(Theorems \ref{thm12},  \ref{thm14} and  \ref{thm16}, respectively)
are based on the proficient graph structure analysis.  
In 2015, F\"{u}redi  \cite{Fur2015} provided a concise and 
contemporary proof  
of Theorem \ref{thm12} 
by applying a result of Erd\H{o}s, Frankl and R\"{o}dl \cite{EFR1986}, 
which is a direct consequence of the Szemer\'{e}di regularity lemma. 
In 2021, Liu \cite{Liu2021} presented  short proofs to  the stability results of two extremal hypergraph problems.

\medskip 
In this paper, we study various stability theorems.  
Motivated by the works of F\"{u}redi and Liu, 
we shall present an alternative proof of the spectral stability theorem.  
Our method shows that the spectral stability theorem can be deduced 
from the Erd\H{o}s--Simonovits stability theorem. 
Moreover, we shall prove that 
the clique stability theorem of Ma and Qiu can also be deduced from 
the Erd\H{o}s--Simonovits stability theorem. 
Although the spectral version and clique version seem more 
stronger than the classical edge version, 
we show that these three different versions   
are  equivalent.

\section{Preliminaries}

In this section, we shall introduce some useful lemmas.  
The first lemma is 
a celebrated result of Erd\H{o}s, Frankl and R\"{o}dl \cite[Theorem 1.5]{EFR1986}. 
It is a direct consequence of 
the Szemer\'{e}di regularity lemma and graph embedding lemma.

\begin{lemma}[Erd\H{o}s--Frankl--R\"{o}dl, 1986] \label{lemEFR}
Let $F$ be a graph and $\varepsilon >0$ be an arbitrary  number. 
There is  $n_0$ such that 
if $n\ge n_0$ and $G$ is an $n$-vertex $F$-free graph, 
then we can remove at most $\varepsilon n^2$ edges from $G$ 
so that the remaining graph is $K_r$-free, where $r=\chi (F)$.  
\end{lemma}

 Lemma \ref{lemEFR} can be extended easily 
to hypergraphs as an implication of hypergraph removal lemma.   
For a $k$-uniform hypergraph $H$, 
{\it the $s$-blow-up of $H$}, denoted by $H(s)$,  
is a $k$-uniform hypergraph 
obtained from $H$ by replacing each vertex $v\in V(H)$ by 
an independent set $I_v$ of $s$ vertices, 
say $x_v^1,x_v^2,\ldots ,x_v^s$.
If $\{v_1,v_2,\ldots ,v_k\}$ is an edge of $H$, then 
we define  
$\{x_{v_1}^{a_1},x_{v_2}^{a_2},\ldots ,x_{v_k}^{a_k}\}$ 
as an edge of $H(s)$
for all $1\le a_1,\ldots ,a_k \le s$. That is, $H(s)$ is obtained from $H$ by 
replacing each edge of $H$ 
by a complete $k$-partite $k$-uniform hypergraph with each part of size $s$. 

  \begin{theorem}  \label{thm6599}
 Let $H$ be a  $k$-uniform hypergraph. 
 For every $\varepsilon >0$ and  $s\ge 1$, 
 there exists $n_0=n_0(H,\varepsilon, s)$ such that 
 if $G$ is an $H(s)$-free $k$-uniform hypergraph on 
 $n\ge n_0$ vertices, then 
 we can remove at most $\varepsilon n^k$ edges from 
 $G$ so that the remaining hypergraph is $H$-free. 
 \end{theorem}

  \begin{proof}
 The proof is a well application of 
 the hypergraph removal lemma, 
 which states that 
 for every $k$-graph $H$ and every $\varepsilon >0$, 
there is $\delta =\delta_H (\varepsilon )>0$ 
such that any graph on $n$ vertices with at most $\delta n^{h}$ copies of $H$ can be 
made $H$-free by removing at most $\varepsilon n^k$ edges. 
 Let $G$ be an $H(s)$-free $k$-uniform hypergraph. 
 First of all, we can see that the number of copies of 
 $H$  is at most $o(n^h)$. Otherwise, applying a theorem of Erd\H{o}s for $k$-partite $k$-graph, 
 we can prove that  $\Omega (n^h) $ copies of $H$ in $G$ lead to a copy of 
 $H(s)$ for sufficiently large $n$; see, e.g., \cite[Theorem 2.2]{Keevash11}. 
 Next, applying the hypergraph removal lemma, 
 we can remove at most $o(n^k)$ edges from $H$ 
 to make it  $H$-free. 
 \end{proof}

The following lemma  describes a relationship 
between the number of copies of $K_s$ and $K_t$ in a $K_{r+1}$-free 
graph, where $s,t$ are two integers less than $r+1$. 

\begin{lemma}[Khadzhiivanov \cite{Kha1977}; S\'{o}s--Straus \cite{SS1982}] \label{lemFR}
Let $G$ be a $K_{r+1}$-free graph on $n$ vertices. 
For every $i\in [r]$, let $k_i$ denote the number of copies of $K_i$ 
in $G$. Then 
\[  \left( \frac{k_r}{{r \choose r}}\right)^{1/r} 
\le  \left( \frac{k_{r-1}}{{r \choose r-1}}\right)^{1/(r-1)} 
\le \cdots \le \left( \frac{k_2}{{r \choose 2}}\right)^{1/2} 
\le \left( \frac{k_1}{{r \choose 1}}\right)^{1/1}.  \]
\end{lemma}

After finishing this paper, Nikiforov \cite{Nikpri2021} told us that 
Lemma \ref{lemFR} was rediscovered independently by many people in the literature; 
see \cite{Niki2013} for a short survey 
including a complete
analytical proof  and determining the cases of equality.

\medskip 
Let $G$ be a graph on $n$ vertices with $m$ edges. 
Let $A$ be the adjacency matrix of $G$. 
It is well-known that $2m/n \le \lambda (G)\le \sqrt{2m}$, 
which is guaranteed by the Rayleigh inequality $\lambda (G)\ge \bm{1}^TA\bm{1}/ (\bm{1}^T\bm{1})=2m/n$ and $\lambda(G)^2 \le \sum_{i=1}^n \lambda_i^2 =\mathrm{tr}(A^2) =\sum_{i=1}^n d_i=2m$. 
Moreover, we can easily show that  
\begin{equation}  \label{equp}
\lambda (G)\le \sqrt{2m\left(1-\frac{1}{n}\right)}. 
\end{equation} 
 Indeed, we  
 observe first that $\sum_{i=1}^n \lambda_i= \tr (A) =0$ and 
 $\sum_{i=1}^n \lambda_i^2 =\tr (A^2)=\sum_{i=1}^n d_i=2m$. 
 Applying the Cauchy--Schwarz inequality, we obtain 
 $(2m-\lambda_1^2)(n-1)=
 (\lambda_2^2 + \cdots + \lambda_n^2) (n-1) 
 \ge (\lambda_2 + \cdots + \lambda_n)^2=\lambda_1^2$, 
 which implies $\lambda_1^2\le 2m(1-\frac{1}{n})$.

\medskip 
In 2002, Nikiforov \cite{Niki2002cpc}  proved a further improvement 
by applying the Motzkin--Straus theorem. 
The conditions when equality holds in Lemma \ref{thmnikiforov} was later 
determined in \cite{Niki2009jctb}.

\begin{lemma}[Nikiforov, 2002] \label{thmnikiforov}
Let  $G$ be a graph with $m$ edges.  Let $\omega$ be the clique number of $G$, the size of a largest complete subgraph of $G$. 
Then 
\begin{equation*}
  \lambda (G) \le \sqrt{2m\Bigl( 1-\frac{1}{\omega}\Bigr)}. 
\end{equation*}
Moreover,  the equality holds 
if and only if 
$G$ is a complete bipartite graph for $r=2$, 
or a complete regular $r$-partite graph for $r\ge 3$ 
by adding some isolated vertices.
\end{lemma}

\noindent 
{\bf Remark.} 
Lemma \ref{thmnikiforov} implies that 
if $G$ is a $K_{r+1}$-free graph with $m$ edges, 
then $\lambda (G)^2 \le {2m(1-\frac{1}{r})}$.  
On the one hand, combining with $2m/n \le \lambda (G)$, 
we get the Tur\'{a}n theorem $m\le 
(1-\frac{1}{r}) \frac{n^2}{2}$. 
On the other hand, using $m\le 
(1-\frac{1}{r}) \frac{n^2}{2}$, we have $\lambda (G) \le (1-\frac{1}{r})n$.

\section{Main results}

\subsection{Alternative proofs}

Nikiforov's proof of Theorem \ref{thm13} 
depends on two of his important theorems: 
the first theorem investigated the relation between 
the number of cliques and the spectral radisu \cite{BN2007jctb}, 
the second theorem states that 
every graph with many $r$-cliques contains a large complete 
$r$-partite subgraph \cite{Niki2008blms}. 
The proof of Alon and Shikhelman for Theorem \ref{thm15} applied 
 the graph removal lemma and needed a proposition \cite[Proposition 2.1]{AS2016} 
 to show that the number of copies of $K_{r+1}$ in $G$ is at most $o(n^{r+1})$.

\begin{proof}[Alternative proof of Theorems \ref{thm13} and \ref{thm15}]
Since $\chi (F)=r+1$, we know that the Tur\'{a}n graph 
$T_r(n)$ is $F$-free. Moreover, we have $\lambda (T_r(n)) 
\ge (1-\frac{1}{r})n- \frac{r}{4n}$ and for $2\le s\le r$, 
$k_s(T_r(n))=\sum_{0\le i_1\le \cdots \le i_s\le r-1} 
\prod_{t=1}^s \lfloor \frac{n+i_t}{r}\rfloor = 
{r \choose s} (\frac{n}{r})^s - o(n^s)$. 
Thus the lower bounds in Theorems \ref{thm13} and \ref{thm15} can be  witnessed by taking $T_r(n)$ as an example. 
Now, assume that $G$ is an $n$-vertex $F$-free graph. 
By Lemma \ref{lemEFR}, we can remove $o(n^2)$ edges from 
$G$ and get a new graph $G^*$ which is $K_{r+1}$-free. 

On the one hand, we claim that the removal of $o(n^2)$ edges from $G$ can only decrease $\lambda (G)$ by at most 
$o(n)$. Indeed, the Rayleigh inequality gives 
$\lambda (G) \le \lambda (G^*) + \lambda (G\setminus G^*) $ and 
the inequality (\ref{equp}) implies $\lambda (G\setminus G^*) 
\le \sqrt{2e(G\setminus G^*)} = o(n)$. 
Since $G^*$ is $K_{r+1}$-free, the Nikiforov result (\ref{eqeq2}) 
implies $\lambda (G^*) \le \mathrm{ex}_{\lambda}(n,K_{r+1}) 
=\lambda (T_r(n))$. Thus, we get $\lambda (G) \le \lambda (T_r(n)) + o(n)$. This completes the proof of Theorem \ref{thm13}. 

On the other hand, each edge of $G$ is contained in at most 
${n-2 \choose s-2}$ copies of $K_s$. This implies that the removal of $o(n^2)$ edges from $G$ 
can merely remove at most $o(n^2){n-2 \choose s-2}=o(n^s)$ copies of 
$K_s$, thus $k_s(G^*)\ge k_s(G)-o(n^s)$. 
Note that $G^*$ is $K_{r+1}$-free, hence the Zykov result (\ref{eqeq4}) 
gives $k_s(G^*)\le \mathrm{ex}(n,K_s,K_{r+1}) =k_s(T_r(n))$. 
Therefore, we obtain $k_s(G)\le k_s(G^*) +o(n^s)\le 
k_s(T_r(n)) + o(n^s)$. This completes the proof of Theorem \ref{thm15}. 
\end{proof}

Being an interesting property of extremal problems, 
the spectral stability theorem 
also gives rise to a surprisingly useful tool for proving the exact values of spectral 
Tur\'{a}n extremal problems. We note that the original proof of Theorem \ref{thm14} 
presented in \cite{Niki2009jgt} relies heavily on  
a spectral stability result of large joint \cite[Theorem 4]{Niki2009ejc} as well as 
a renowned result in \cite[Theorem 1]{Niki2008blms}. 
But the proof of the spectral stability result of large joint 
stated in \cite{Niki2009ejc} seems complicated and 
needs a series of results from Nikiforov's works 
in the order 
\cite{BN2007jctb, Niki2008blms, BN2008dm, Niki2010dm, Niki2008}. 
In the sequel, we shall provide a new short proof of Theorem \ref{thm14}. 
The line of proofs are quite different. 

\begin{proof}[Alternative proof of Theorem \ref{thm14}]
Recall that $F$ is a graph with $\chi (F)=r+1\ge 3$. 
Let $\varepsilon >0$ be a small fixed  number. 
Let $G$ be an $F$-free graph on $n$ vertices such that 
$\lambda (G) \ge (1-\frac{1}{r} -\delta )n$, 
where $n$ is a large enough number and $\delta$ is a small enough number 
determined later. 
Applying the Simonovits stability theorem (Theorem \ref{thm12}) to $K_{r+1}$, 
we know that there exist $\delta (K_{r+1}, \frac{\varepsilon}{2} )>0$ 
and $n_0=n_0(K_{r+1}, \frac{\varepsilon}{2} )$ such that 
if $H$ is an arbitrary graph on $n\ge n_0$ vertices satisfying that 
 $H$ is $K_{r+1}$-free and  
$e(H)\ge \left(1- \frac{1}{r} - \delta (K_{r+1}, \frac{\varepsilon}{2}) \right) \frac{n^2}{2}$, 
 then the edit distance $d(H,T_r(n))\le \frac{\varepsilon}{2} n^2$.

Let $\varepsilon_0 \in (0, \frac{\varepsilon}{2}]$ be a sufficiently small number determined later. 
Since $G$ is $F$-free, applying Lemma \ref{lemEFR}, 
we obtain an integer $n_1=n_1(F,\varepsilon_0)$ such that 
if $n\ge n_1$,  
then we  get a $K_{r+1}$-free subgraph  by removing at most 
$\varepsilon_0 n^2$ edges from $G$. 
We denote the resulting subgraph by $G^*$. 
Moreover, the Rayleigh formula and inequality (\ref{equp}) give  
$ \lambda (G) \le 
\lambda (G^*) + \lambda (G\setminus G^*)
<  \lambda (G^*) +  \sqrt{2e(G\setminus G^*)} $. 
This implies that $\lambda (G^*)> \lambda (G) -\sqrt{2\varepsilon_0 }n
\ge (1-\frac{1}{r}-\delta -\sqrt{2\varepsilon_0})n$. 
Recall that $G^*$ is $K_{r+1}$-free. 
Applying the remark of Lemma \ref{thmnikiforov} to  $G^*$, 
we have $\lambda (G^*)^2 \le (1-\frac{1}{r})2e(G^*)$, which 
implies $e(G^*)\ge \left(1-\frac{1}{r}-2(\delta +\sqrt{2\varepsilon_0}) 
\right)\frac{n^2}{2}$. 
We  choose sufficiently small $\varepsilon_0>0$, and 
sufficiently large $n\ge \max\{n_0,n_1\}$, then we can choose 
small $\delta >0$ such that $2(\delta +\sqrt{2\varepsilon_0}) 
\le \delta (K_{r+1},\frac{\varepsilon}{2})$. The Simonovits stability theorem 
gives $d(G^*,T_r(n))\le \frac{\varepsilon}{2} n^2$. 
Thus $d(G,T_r(n))\le d(G,G^*) + d(G^*, T_r(n)) 
\le \varepsilon n^2$. 
\end{proof}

We now give a short proof of Theorem \ref{thm16}. 

\begin{proof}[Alternative proof of Theorem \ref{thm16}] 
Let $G$ be an $F$-free graph on $n$ vertices such that 
$k_s (G) \ge  {r \choose s} (\frac{n}{r})^s - \delta n^s$, 
where $n$ is a large enough number and $\delta$ is a small enough number 
determined later. 
By Theorem \ref{thm12}, we know that for every $\varepsilon >0$, 
there exist $\delta (K_{r+1}, \frac{\varepsilon}{2} )>0$ and $n_0(K_{r+1}, \frac{\varepsilon}{2} )$ such that 
if $H$ is a graph on $n\ge n_0$ vertices, 
 and $H$ is $K_{r+1}$-free such that 
$e(H)\ge \left(1- \frac{1}{r} - \delta (K_{r+1}, \frac{\varepsilon}{2} )
 \right) \frac{n^2}{2}$, 
 then the edit distance $d(H,T_r(n))\le \frac{\varepsilon}{2} n^2$.

Let $\varepsilon_0 >0$ be a sufficiently small number. 
Since $G$ is $F$-free, applying Lemma \ref{lemEFR}, 
we obtain an integer $n_1=n_1(F,\varepsilon_0)$ such that 
if $n\ge n_1$,  
then we  get a $K_{r+1}$-free subgraph  by removing at most 
$\varepsilon_0 n^2$ edges from $G$. 
We denote the resulting subgraph by $G^*$. 
Moreover, every edge of $G$ is contained in at most 
${n-2 \choose s-2}$ copies of $K_s$. Thus 
the removal of $\varepsilon_0 n^2$ edges from $G$ can only destroy 
at most $\varepsilon_0 n^2 {n-2 \choose s-2}<\varepsilon_0 n^s$ copies of $K_s$ in $G$. 
This implies that $k_s (G^*)\ge k_s (G) - \varepsilon_0 n^s 
\ge {r \choose s} (\frac{n}{r})^s - \delta n^s -\varepsilon_0 n^s$. 
Note that $G^*$ is $K_{r+1}$-free. 
Applying Lemma \ref{lemFR} to  $G^*$, 
we have $(k_2 (G^*)/ {r \choose 2})^{1/2} \ge (k_s(G^*) /{r \choose s})^{1/s}$, which 
implies $e(G^*)\ge 
{r \choose 2} \left( 
(\frac{n}{r})^s -(\delta  + \varepsilon_0)n^s/ 
{r \choose s} \right)^{2/s}
\ge (1-\frac{1}{r})\frac{n^2}{2} - t$, where 
$t={r \choose 2} \left( (\delta  + \varepsilon_0)n^s/ 
{r \choose s} \right)^{2/s}$. 
We  choose sufficiently small $\varepsilon_0>0$, and 
sufficiently large $n\ge \max\{n_0,n_1\}$, then we can choose 
 $\delta >0$ small enough such that $t \le \delta (K_{r+1},\frac{\varepsilon}{2})\frac{n^2}{2}$. 
Thus  the Simonovits stability theorem implies 
$d(G,T_r(n))\le d(G,G^*) + d(G^*,T_r(n)) \le \varepsilon n^2$. 
\end{proof}

\subsection{Revisiting color-critical graphs}

\label{Sec3.2}
Let $e$ be an edge of graph $F$. 
We say that $e$ is a color-critical edge of $F$ if  
$\chi (F-e)<\chi (F)$. We say that $F$ is color-critical 
if $F$ contains a color-critical edge. 
There are many graphs that are color-critical. 
For instance, the following graphs are color-critical. 
Every edge of the complete graph $K_{r+1}$ is a color-critical edge. 
Then $K_{r+1}$ is color-critical with $\chi (K_{r+1})=r+1$. 
Every edge of the odd cycle $C_{2k+1}$ is a color-critical edge. 
So  $C_{2k+1}$ is color-critical with $\chi (C_{2k+1})=3$. 
Let $W_{n}=K_1 \vee C_{n-1}$ be the wheel graph on $n$ vertices, 
a vertex that joins all vertices of  $C_{n-1}$. 
The even wheel graph $W_{2k}$ is color-critical and $\chi (W_{2k})=4$. 
However, we can check that the odd wheel graph $W_{2k+1}$ is not color-critical and $\chi (W_{2k+1})=3$. 
Let $B_k=K_2 \vee I_{k-2}$ be the book graph, 
that is, $k$ triangles sharing a common edge. Then 
$B_k$ is color-critical and $\chi (B_k)=3$.

In 1966, Simonovits \cite{Sim1966} proved 
a celebrated result, which gives 
the exact Tur\'{a}n number for all color-critical graphs. 

\begin{theorem}[Simonovits, 1966] \label{thmSim66}
If $F$ is a  graph with a critical edge and $\chi (F)=r+1$ where $r\ge 2$,  
then there exists an $n_0=n_0(F)$ 
such that  
\[    \mathrm{ex}(n,F)=e(T_r(n))  \]
 holds for all $n\ge n_0$, and 
the unique extremal graph is the Tur\'{a}n graph $T_r(n)$. 
\end{theorem}

In 2009, Nikiforov \cite[Theorem 2]{Niki2009ejc} proved the 
corresponding result in terms of the spectral radius.

\begin{theorem}[Nikiforov, 2009] \label{thmcri}
If $F$ is a  graph with a critical edge and $\chi (F)=r+1$ where $r\ge 2$, 
then there exists an $n_0=n_0(F)$ 
such that  
\[  \mathrm{ex}_{\lambda}(n,F)=\lambda (T_r(n))  \]
 holds for all $n\ge n_0$, and 
the unique extremal graph is $T_r(n)$. 
\end{theorem}

Note that Theorem \ref{thmcri} implies Theorem \ref{thmSim66} 
by applying (\ref{eqeqq3}). Indeed, assume that $F$ is a color-critical 
graph with $\chi (F)=r+1$ and $G$ is  $F$-free, Theorem \ref{thmcri} implies that 
for sufficiently large $n$, we have 
$\lambda (G)\le \lambda (T_r(n))$, which together with 
(\ref{eqeqq3}) yields $e(G)\le e(T_r(n))$.

Correspondingly, it was proved by Ma and Qiu \cite{MQ2020} that 
for a color-critical graph $F$ with $\chi (F)=r+1$ and sufficiently large $n$, 
the Tur\'{a}n graph $T_r(n)$ is the unique graph attaining the maximum 
number of copies of $K_s$ in an $n$-vertex $F$-free graph.

In what follows, we escape from the framework of the proof of Nikiforov
and  give an alternative proof 
of Theorem \ref{thmcri} by applying the spectral stability theorem 
(Theorem \ref{thm14}). 
The proof of Nikiforov \cite{Niki2009ejc} 
relies heavily on a series of works stated in the order of \cite{Niki2007laa2,Niki2010dm,Niki2008,BN2007jctb,Niki2008blms}.
Our proof is more transparent and straightforward, we shall  
use some significant ideas and techniques of the 
Szemer\'{e}di regularity lemma although we do not apply 
the regularity lemma directly.

Let $G$ be an $n$-vertex graph and $G'$ 
be an $r$-partite subgraph of $G$ with partition 
$V(G')=U_1\cup U_2 \cup \cdots \cup U_r$. 
Let $\varepsilon >0$ be a sufficiently small number.  
We say that $G'$ is $\varepsilon$-almost complete 
if for any $v\in U_i$ and any $j\neq i$, 
we have $|N(v) \cap U_j| \ge |U_j| -\varepsilon n$.  
In other words, the number of non-neighbors of $v$ in $U_j$ 
is at most $\varepsilon n$.

\begin{lemma} \label{lemcri}
Let $F$ be a color-critical  graph  $\chi (F)=r+1$ where $r\ge 2$. 
Let $\varepsilon >0$ be small enough. 
Let $G$ be an $n$-vertex $F$-free graph 
with $n$ large enough. 
If $G$ contains an $\varepsilon$-almost complete 
$r$-partite subgraph $G'$ with 
$V(G')=U_1\cup U_2 \cup \cdots \cup U_r$, and large enough $|U_i|$,  
then \\
(i) All parts $U_1,U_2,\ldots ,U_r$  are independent sets in $G$. \\
(ii) For every $x\in V(G)\setminus V(G')$, 
there exists a vertex set $U_i$ such that 
$x$ has at most $\varepsilon rt n$ neighbors in $U_i$,  
where $t$ is the number of vertices of $F$.  
\end{lemma}

\begin{proof}
(i) Since  $F$ is a  graph with a critical edge and $\chi (F)=r+1$, 
there is an edge $\{x,y\}\in E(F)$  
such that $F\setminus \{x,y\}$ is a subgraph of 
the complete $r$-partite graph $K_{t,t,\ldots ,t}=K_r(t)$. 
Without loss of generality, we may assume on the contrary that 
$U_1$ contains an edge $\{u,v\}$. 
We choose a vertex set $T_1 \subseteq U_1$ satisfying $|T_1|=t$ and 
$\{u,v\} \subseteq T_1$.  Suppose we have obtained 
sets $T_1,T_2,\ldots ,T_i$ such that 
$|T_i|=t$ and $T_1,T_2,\ldots ,T_i$ form a complete $i$-partite graph 
$K_{i}(t)$.  Note that $G'$ is $\varepsilon$-almost complete, 
so every $x\in T_1\cup \cdots \cup T_i$ misses at most $\varepsilon 
n$ vertices in $U_{i+1}$, i.e., excepting at most 
$it\varepsilon n$ vertices of $U_{i+1}$, the remaining vertices 
of $U_{i+1}$ are adjacent to every vertex of $T_1\cup \cdots \cup T_i$. 
Observe that $|U_{i+1}| - it \varepsilon n \ge t$. 
Hence there exists a set $T_{i+1} \subseteq U_{i+1}$ with 
$|T_{i+1}| =t $ 
such that the vertex sets $T_1, \ldots , T_i,T_{i+1}$ 
form a complete $(i+1)$-partite graph $K_{i+1}(t)$. 
We can proceed this operation until we find $r$ sets 
$T_1,T_2,\ldots ,T_{r}$ satisfying $|T_i|=t$ and they 
form a complete $r$-partite graph. Note that $\{u,v\}$ is an edge in 
$T_1$, it corresponds to the edge $\{x,y\}$ of $F$. 
Thus $F $ is contained in $G'$, and so it is contained in $G$, 
a contradiction. 

(ii) Suppose that $x$ has at least $\varepsilon rtn$ neighbors 
in  set $U_i$ for every $i=1,2,\ldots ,r$. 
Since $F$ is a fixed graph on $t$ vertices, 
we know that $\varepsilon rtn \ge t$ holds for fixed small $\varepsilon $ and 
then sufficiently large $n$. 
Setting $U_i' =N(x)\cap U_i$ for every $i=1,2,\ldots ,r$. 
Firstly, we choose 
a set $T_1 \subseteq U_1'$ with $|T_1| =t$. 
Suppose that we have find sets $T_1,T_2,\ldots ,T_i$ 
satisfying $T_i \subseteq U_i', |T_i|=t$ and 
$T_1,T_2,\ldots ,T_i$ form a complete $i$-partite graph $K_i(t)$.  
Observe that every $v\in T_1 \cup \cdots \cup T_i$ misses at most 
$\varepsilon n $ vertices of $U_{i+1}$, so it misses at most $\varepsilon n$ vertices of $U_{i+1}'$. 
Note that $|U_{i+1}'| \ge rt \varepsilon n$, hence for sufficiently large $n$, 
we can find  a set $T_{i+1} \subseteq U_{i+1}'$ such that $|T_{i+1}|=t$ and 
every vertex of $T_{i+1}$ is adjacent to every vertex of $T_1\cup \cdots \cup T_i$. By repeating this process, 
 we can find $r$ sets $T_1,T_2,\ldots ,T_r $ such that 
$T_i \subseteq U_i', |T_i|=t$ and 
$T_1,T_2,\ldots ,T_r$ form a complete $r$-partite graph.  
Observe that $F$ is contained in the 
complete $(r+1)$-partite graph formed by $\{x\},T_1,\ldots ,T_r$, 
so it is contained in $G$, 
a contradiction. 
\end{proof}

We now start our new proof of Theorem \ref{thmcri}.

\begin{proof}[Alternative proof of Theorem \ref{thmcri}]
Let $F$ be a graph with a critical edge and $\chi (F)=r+1$. 
Assume that $G$ is an $n$-vertex $F$-free graph with 
$\lambda (G)\ge \lambda (T_r(n))$, our goal is to show 
$G=T_r(n)$.  
Set $\varepsilon >0$ as a sufficiently small constant.  
First of all, we claim that $\delta (G) \ge (1-\frac{1}{r} -\varepsilon )n$. 
{\it Otherwise, if there exists a vertex with degree less than $(1-\frac{1}{r} - \varepsilon )n$, } 
then by a standard argument of successively-vertex deletion \cite[Theorem 5]{Niki2008}, 
we  obtain  a subgraph $G_1 $ of $G$ 
satisfying $|G_1|=n_1\ge n/2$ vertices  and one of the following conditions: \\ 
(a) $\lambda (G_1) > (1-\frac{1}{r} + \frac{\varepsilon}{4})n_1$; or \\ 
(b) $\delta (G_1) \ge (1- \frac{1}{r} - \varepsilon )n_1$ and
  $\lambda (G_1)>  \lambda (T_r(n_1))$. 
  
  If Condition (a) happens, then we can see that $G_1$ contains $F$ as a subgraph 
  by applying  Theorem \ref{thm13} for sufficiently large $n$. 
 In what follows, we consider Condition (b). 
 In other words, $G$ has a subgraph $G_1$ on $n_1>n/2$ vertices with  
 $\lambda (G_1)> \lambda (T_r(n_1))$ and 
 the minimum degree $\delta (G_1)\ge (1-\frac{1}{r} - \varepsilon )n_1$.

 Note that $\lambda (T_r(n_1))\ge (1-\frac{1}{r})n_1 - \frac{r}{4n_1}$. 
 By the spectral stability Theorem \ref{thm14}, 
we have $d(G_1,T_r(n_1))\le \varepsilon n_1^2$. 
Hence there exists an $r$-partition  of the vertex set of $G_1$, 
say $V(G_1)=V_1 \cup V_2 \cup \cdots \cup V_r$ such that 
$|V_i|=\lfloor n_1/r\rfloor$ or $\lceil  n_1/r\rceil$ for every $i=1,2,\ldots ,r$, 
and 
$ \sum_{i=1}^r e(G_1[V_i]) + e(T_r(n_1)) - e(G_1[V_1,V_2,\ldots ,V_r]) \le \varepsilon n_1^2$, 
where $G_1[V_i]$ is the  subgraph of 
$G_1$ induced by $V_i$ and $G_1[V_1,V_2,\ldots ,V_r]$ 
is the $r$-partite subgraph of $G_1$ formed between sets $V_1,V_2, \ldots ,V_r$. Note that $e(T_r(n_1)) - e(G_1[V_1,V_2,\ldots ,V_r])$ 
is the number of edges of $T_r(n)$ missing from  $G_1$, 
i.e., the pair $\{x,y\}$ with $x\in V_i,y\in V_j,i\neq j$ and 
$\{x,y\}\notin E(G_1)$. 
For every $i=1,2,\ldots ,r$, we define 
$B_i$ as the set of vertices of $V_i$ missing at least $\sqrt{\varepsilon }n_1$ 
edges from some $V_j$, that is, $B_i=\{v\in V_i: |N(v)\cap V_j| \le |V_j|- \sqrt{\varepsilon}n_1 ~\text{for some $j\neq i$}\}$. 
 We call such vertex a  bad vertex and 
denote $B=\cup_{i=1}^r B_i$.  Then $|B|\le \frac{2\varepsilon n_1^2}{
\sqrt{\varepsilon}n_1}=2\sqrt{\varepsilon}n_1$.  
Let $U_i=V_i \setminus B_i$ for every $i=1,\ldots ,r$. 
Then $|U_i|\ge |V_i| - |B| > \frac{n_1}{r} - 3\sqrt{\varepsilon}n_1$. 
For every $v\in U_i$ and $j\neq i$, 
we have $|N(v)\cap U_j|\ge |N(v)\cap V_j| - |B| 
\ge (|V_j| - \sqrt{\varepsilon}n_1 ) -2 \sqrt{\varepsilon}n_1 
\ge |U_j| - 3\sqrt{\varepsilon}n_1  $. 
Thus $G_1[U_1,U_2,\ldots ,U_r]$ is a $3\sqrt{\varepsilon}$-almost 
complete $r$-partite subgraph of $G_1$. 
By (i) of Lemma \ref{lemcri}, we know that 
$U_1,U_2,\ldots ,U_r$ are independent sets. 

If $B$ is non-empty, 
then for each bad vertex $x\in B$, by (ii) of Lemma \ref{lemcri}, 
there exists a set $U_i$ such that $|N(x)\cap U_i| \le 
3\sqrt{\varepsilon} rtn_1$. 
Now we add the bad vertex $x$ into the set $U_i$, 
and get a new $r$-partition, say $U_1', \ldots ,U_i', \ldots ,U_r'$, 
where $U_i'=U_i \cup \{x\}$ and other $U_j'=U_j$ for $j\neq i$.   
We claim that the new partition $G_1[U_1', \ldots ,U_r']$ 
is $(7 \sqrt{\varepsilon}rt)$-almost complete. 
Note that $\delta (G_1) \ge (1-\frac{1}{r} - \varepsilon)n_1$, 
so we have 
$|N(x) \cap (\cup_{j\neq i} U_j)| 
\ge d(x) - |B| - |N(x)\cap U_i|  
 \ge (1-\frac{1}{r} - \varepsilon)n_1 - 2\sqrt{\varepsilon}n_1 
- 3\sqrt{\varepsilon}rtn_1   
 \ge (1-\frac{1}{r}) n_1 - 6\sqrt{\varepsilon} rtn_1$. 
 Note that $|\cup_{j\neq i} U_j| \le |\cup_{j\neq i} V_j| \le n- 
 \lfloor \frac{n_1}{r}\rfloor < (1-\frac{1}{r})n_1 + \sqrt{\varepsilon}rtn_1$ 
 for sufficiently large $n_1$.   Thus we obtain 
 $|N(x)\cap (\cup_{j\neq i} U_j)| \ge |\cup_{j\neq i} U_j| - 
 7 \sqrt{\varepsilon}rtn_1$. 
 So the number of non-neighbors of $x$ in set $\cup_{j\neq i} U_j$ 
 is at most $7 \sqrt{\varepsilon}rtn_1$, 
 which implies that the number of non-neighbors of $x$ in 
each $U_j$ with $j\neq i$ is at most $7 \sqrt{\varepsilon}rtn_1$. 
Thus we get $|N(x)\cap U_j| \ge |U_j| - 7 \sqrt{\varepsilon}rtn_1$ 
for every $j\neq i$. So we complete the proof of our claim. 
By (i) of Lemma \ref{lemcri}, we know that 
$U_1',U_2',\ldots ,U_r'$ are still independent sets. 

If $B\setminus \{x\}$ is non-empty, 
then for each bad vertex $y\in B\setminus \{x\}$, by (ii) of Lemma \ref{lemcri}, 
there exists a set $U_i'$ such that $|N(y)\cap U_i'| \le 
(7\sqrt{\varepsilon} rt) \cdot rtn_1$. 
We add the vertex $y$ to the set $U_i'$, 
and obtain a new $r$-partition. 
Repeating the above process, we can keep adding all bad vertices 
of $B$ into our $r$-partition until $B=\varnothing$. 
At the end of our process, we conclude that $G_1$ is 
an $r$-partite graph. Certainly, $G_1$ is $K_{r+1}$-free, the spectral Tur\'{a}n theorem implies $\lambda (G_1) \le \lambda (T_r(n_1))$. 
Recall that Condition (b) states that $\lambda (G_1) > \lambda (T_r(n_1))$, 
so we get a contradiction.

Hence, there are no vertices of $G$ with degree less than $(1-\frac{1}{r}-\varepsilon )n$. 
We conclude that $G$ is an $F$-free graph on $n$ vertices with 
$\lambda (G)\ge \lambda (T_r(n))$ and $\delta (G) \ge (1-\frac{1}{r} -\varepsilon )n$. 
Now, replacing $G_1$ with $G$, and repeating the above  discussions, 
we can show that $G$ is an $r$-partite graph. 
Keeping in mind that $\lambda(G)\ge \lambda (T_r(n))$, 
we obtain that $G$ is a balanced complete $r$-partite graph on 
$n$ vertices. 
\end{proof}

Our proof in above  is a standard graph structure analysis.  
It is worth noting that one can prove Theorem \ref{thmSim66} by modifying slightly the above proof 
of the spectral version. 
Generally speaking, the stability method has two steps. 
First one need to prove a stability theorem, i.e., any construction of close to 
maximum size is structurally close to the conjectured extremal graph. 
Armed with this approximate  structure, 
we can consider any supposed better construction as being obtained from the  extremal example by introducing a small number of  imperfections into the structure. 
The second step is to analyze any possible imperfection and show that it must  lead to a suboptimal configuration, so in fact the conjectured extremal example must be optimal; 
see \cite{Keevash11,MY2018} for similar examples.

\subsection{Making $K_{r+1}$-free graphs $r$-partite}

In this subsection, we shall study the problem of 
how many edges needed to be removed in a $K_{r+1}$-free graph to 
make its being $r$-partite. 
For integer $r\ge 2$, let $D_r(G)$ denote the minimum number of edges 
which need to be removed to make $G$ being $r$-partite. 
For cliques, 
the Erd\H{o}s--Simonovits theorem \cite{Sim1966} 
 states that for every $\varepsilon >0$, there exist $\delta >0$ and $n_0$ 
such that if $G$ is a $K_{r+1}$-free graph on $n\ge n_0$ vertices and 
$e(G)\ge e(T_r(n)) -\delta n^2$, then $D_r(G)\le \varepsilon n^2$. 
In 2015, F\"{u}redi \cite{Fur2015} provided an elegant proof of the result  
that every $K_{r+1}$-free  graph $G$ on $n$ vertices 
with at least $e(T_r(n))- t$ edges satisfies $D_r(G)\le t$. 
This provided a quantitative improvement of 
the Erd\H{o}s--Simonovits theorem. 
Very recently, Balogh, Clristian, Lavrov, Lidick\'{y} and Pfender \cite{BCLLP2021}
 determined asymptotically a sharp bound on the number of edges that are needed 
 for small $t\in \mathbb{N}^*$.

\begin{theorem}[Balogh--Clristian--Lavrov--Lidick\'{y}--Pfender, 2021] \label{thmBal}
Let $r\ge 2$ be an integer. 
For all $n\ge 3r^2$ and $0\le \delta \le 10^{-7}r^{-12}$, 
the following holds. If $G$ is a $K_{r+1}$-free graph on 
$n$ vertices with $e(G)\ge e(T_r(n)) - \delta n^2$, 
then $D_r(G)\le \bigl( \frac{2r}{3\sqrt{3}} + o_{\delta}(1) \bigr) 
\delta^{3/2}n^2$, 
where $o_{\delta}(1)$ is a term converging to $0$ as $\delta$ 
tending to $0$. 
\end{theorem}

In the sequel, we  present the spectral version and 
clique version of Theorem \ref{thmBal}. 

\begin{theorem} \label{thmthm35}
Let $r\ge 2, n\ge 3r^2$ and $0\le \delta \le 10^{-7}r^{-12}$. 
If $G$ is an $n$-vertex $K_{r+1}$-free graph such that 
$\lambda(G)\ge (1- \frac{1}{r} - \delta) n$, 
then $D_r(G)\le \bigl( \frac{2r}{3\sqrt{3}} + o_{\delta}(1) \bigr) 
\delta^{3/2}n^2$. 
\end{theorem} 

\begin{proof}
Since $G$ is $K_{r+1}$-free, 
applying the remark of Lemma \ref{thmnikiforov}, we have 
$\lambda (G)^2 \le (1-\frac{1}{r}) 2e(G)$. Thus we have 
$e(G)\ge \frac{1}{2}\lambda (G)^2/ (1-\frac{1}{r}) \ge  \frac{n^2}{2}
(1-\frac{1}{r} -\delta )^2 / (1-\frac{1}{r}) 
> \frac{n^2}{2}(1-\frac{1}{r}- 2\delta) \ge e(T_r(n)) -\delta n^2$. The desired result  follows immediately from Theorem \ref{thmBal}. 
\end{proof}

\begin{theorem} \label{thm312}
Given $r\ge 2 $ and $s\ge 2$. Let $n$ be  large and $\delta >0$ be sufficiently small. 
If $G$ is a graph on $n$ vertices, 
 and $G$ is $K_{r+1}$-free such that 
$k_s(G)\ge {r\choose s}(\frac{n}{r})^s -\delta n^s$, 
 then  $D_r(G) \le \bigl( \frac{2r}{3\sqrt{3}} + o_{\delta}(1) \bigr)  {r \choose 2}^{3/2}/ {r \choose s}^{3/s} \cdot \delta^{3/s} n^2$. 
\end{theorem}

\begin{proof}
Note that $G$ is $K_{r+1}$-free.  
Applying  Lemma \ref{lemFR}, we get 
$e(G)\ge {r \choose 2} ( {k_s(G)}/{ {r \choose s} })^{2/s}$, 
which together with the assumption yields 
$e(G)\ge {r\choose 2} \left( (\frac{n}{r})^s - \delta n^s/ {r \choose s}  \right)^{2/s} 
\ge {r \choose 2}(\frac{n}{r})^2 -  {r \choose 2}/ {r \choose s}^{2/s} 
\cdot \delta^{2/s} n^2 \ge e(T_r(n)) -  { r \choose 2} / {{r \choose s}^{2/s}} \cdot \delta^{2/s} n^2$. 
By Theorem \ref{thmBal}, we obtain the required result 
$D_r(G) \le \bigl( \frac{2r}{3\sqrt{3}} + o_{\delta}(1) \bigr)  {r \choose 2}^{3/2}/ {r \choose s}^{3/s} \cdot \delta^{3/s} n^2$. 
\end{proof}

\noindent 
{\bf Remark.} 
The Rayleigh formula gives $2e(G)/ n \le \lambda (G)$. 
Thus Theorem \ref{thmthm35} extends  Theorem \ref{thmBal} slightly. 
In particular, the case $s=2$ in Theorem \ref{thm312} reduces to 
Theorem \ref{thmBal}. In addition, 
Theorem \ref{thm312} is an improvement of a 
recent result of Liu \cite[Theorem 4.1]{Liu2021}.

\medskip 

For a positive integer $r\ge 2$, a graph $G$ is said to be $K_{r+1}$-saturated 
(or maximal $K_{r+1}$-free) if it contains no copy of $K_{r+1}$, but the addition of any edge from the complement $\overline{G}$ creates at least one copy of $K_{r+1}$. 
In 2018, Popielarz, Sahasrabudhe and Snyder \cite{PSS2018} 
proved the following stronger stability theorem for $K_{r+1}$-saturated graphs. 

\begin{theorem}[Popielarz--Sahasrabudhe--Snyder, 2018] \label{thmPSS}
Let $r\ge 2$ be an integer. For any $\varepsilon >0$, 
there exist  $\delta >0$ and $n_0$ such that 
if $G$ is a $K_{r+1}$-saturated graph on $n\ge n_0$ vertices 
with $e(G)\ge (1-\frac{1}{r}) \frac{n^2}{2} - \delta n^{\frac{r+1}{r}}$, 
then $G$ contains a complete $r$-partite subgraph on 
$(1-\varepsilon )n$ vertices. 
\end{theorem}

The spectral version and clique version can be  
obtained similarly. 

\begin{theorem} 
Let $r\ge 2$ be an integer. For any $\varepsilon >0$, 
there exist  $\delta >0$ and $n_0$ such that 
if $G$ is a $K_{r+1}$-saturated graph on $n\ge n_0$ vertices 
with $\lambda (G)\ge (1-\frac{1}{r}) n - \delta n^{\frac{1}{r}}$, 
then $G$ contains a complete $r$-partite subgraph on 
$(1-\varepsilon )n$ vertices. 
\end{theorem}

The following theorem  extends Theorem \ref{thmPSS} 
by setting $s=2$. 

\begin{theorem} 
Let $r\ge 2$ and $s\ge 2$ be  integers. For any $\varepsilon >0$, 
there exist  $\delta >0$ and $n_0$ such that 
if $G$ is a $K_{r+1}$-saturated graph on $n\ge n_0$ vertices 
with $k_s (G)\ge {r \choose s} (\frac{n}{r})^s - \delta n^{s\frac{r+1}{2r}}$, 
then $G$ contains a complete $r$-partite subgraph on 
$(1-\varepsilon )n$ vertices. 
\end{theorem}

\subsection{Stability result for the $p$-spectral radius}

The spectral radius of a graph is defined 
as the largest  eigenvalue of its adjacency matrix. 
By the Rayleigh theorem, we know that 
it is also equal to the maximum value of 
$\bm{x}^TA(G)\bm{x}=2\sum_{\{i,j\}\in E(G)} x_ix_j$ over all 
$\bm{x}\in \mathbb{R}^n$ with $|x_1|^2 + \cdots +|x_n|^2=1$. 
The definition of the spectral radius was recently extended 
to the $p$-spectral radius.  We denote the $p$-norm of $\bm{x}$ by  
$\lVert \bm{x}\rVert_p  
=( |x_1|^p + \cdots +|x_n|^p)^{1/p}$. 
More precisely, 
the $p$-spectral radius of graph $G$ is defined as 
\[  \lambda^{(p)} (G) : =\max_{\lVert \bm{x}\rVert_p =1} 
2 \sum_{\{i,j\} \in E(G)} x_ix_j.   \] 
We remark that $\lambda^{(p)}(G)$ is a versatile parameter. 
Indeed, $\lambda^{(1)}(G)$ is known as the Lagrangian function of $G$, 
$\lambda^{(2)}(G)$ is the spectral radius of its adjacency matrix, 
and 
\begin{equation} \label{eqeqeq}
\lim_{p\to +\infty} \lambda^{(p)} (G)=2e(G),
\end{equation}
 which can be guaranteed 
by the following inequality 
\begin{equation} \label{eqps}
  2e(G)n^{-2/p} \le \lambda^{(p)}(G) \le (2e(G))^{1-1/p}.  
  \end{equation}
  To some extent, the $p$-spectral radius can be viewed as a unified extension of the classical 
  spectral radius and the size of a graph. 
In addition, it is worth mentioning that 
if $ 1\le q\le p$, then  $\lambda^{(p)}(G)n^{2/p} \le \lambda^{(q)}(G)n^{2/q}$ 
and $(\lambda^{(p)}(G)/2e(G))^p \le (\lambda^{(q)}(G)/2e(G))^q$; 
see \cite[Proposition 2.13 and 2.14]{Niki2013} for more details.

As commented by Kang and Nikiforov in \cite[p. 3]{KN2014}, 
linear-algebraic methods are irrelevant for the study of 
$\lambda^{(p)}(G)$ in general, and in fact no efficient methods 
are known for it. Thus the study of $\lambda^{(p)}(G)$ 
for $p\neq 2$ is far more complicated than the classical 
spectral radius.

The extremal  function for $p$-spectral radius is given as 
\[  \mathrm{ex}_{\lambda}^{(p)}(n, {F}) := 
\max\{ \lambda^{(p)} (G) : |G|=n ~\text{and $G$ is ${F}$-free}  \}.  \]
To some extent, the proof of  results on the $p$-spectral radius shares some similarities with the usual spectral radius when $p>1$; 
see \cite{Niki2013,KN2014} for extremal problems 
for the $p$-spectral radius.   
In 2014, Kang and Nikiforov \cite{KN2014} extended the Tur\'{a}n theorem 
to the $p$-spectral version for $p>1$. 
They proved that if $G$ is a $K_{r+1}$-free graph on $n$ vertices, 
then $\lambda^{(p)}(G)\le \lambda^{(p)}(T_r(n))$, 
equality holds if and only if $G=T_r(n)$. 
In symbols, we have 
\begin{equation} \label{eqKN}
 \mathrm{ex}_{\lambda}^{(p)} (n,K_{r+1}) =\lambda^{(p)}(T_r(n)). 
\end{equation}

\begin{theorem} \label{thm37}
If $F$ is a graph with chromatic number $\chi (F)=r+1$, 
then for every $p> 1$, 
\[  \mathrm{ex}_{\lambda}^{(p)} (n,F) 
=\lambda^{(p)} (T_r(n)) + o(n^{2- (2/p)})=
\left( 1-\frac{1}{r} + o(1)\right) n^{2- (2/p)} . \] 
\end{theorem} 

Theorem \ref{thm37} extends both Theorem \ref{thm11} and Theorem \ref{thm13} 
by noting (\ref{eqeqeq}). 

\begin{proof}
The Tur\'{a}n graph 
$T_r(n)$ is $r$-partite, so $T_r(n)$ is an $F$-free graph. 
Moreover, by inequality (\ref{eqps}), we have $\lambda^{(p)} (T_r(n)) 
\ge 2e(T_r(n))n^{-2/p} \ge 
 (1-\frac{1}{r})n^{2-(2/p)}- \frac{r}{4n^{2/p}}$. 
 Thus $T_r(n)$ gives the lower bound 
$\mathrm{ex}_{\lambda}^{(p)} (n,F) \ge 
\lambda^{(p)} (T_r(n))\ge (1-\frac{1}{r} +o(1))n^{2-(2/p)}$. 
More precisely, we can obtain  by detailed computation that 
 $\lambda^{(p)} (T_r(n))=(1+O(\frac{1}{n^2}))2e(T_r(n))n^{-2/p} 
 = (1-O(\frac{1}{n^2}))(1-\frac{1}{r})n^{2-2/p} $, 
 where $O(\frac{1}{n^2})$ stands for a positive error term. 

Now, assume that $G$ is an $n$-vertex $F$-free graph. 
By Lemma \ref{lemEFR}, we can remove $o(n^2)$ edges from 
$G$ and get a new graph $G^*$ which is $K_{r+1}$-free. 
We claim that the removal of $o(n^2)$ edges from $G$ can only decrease $\lambda^{(p)} (G)$ by at most 
$o(n^{2-2/p})$. Indeed, we can see from the definition that  
$\lambda^{(p)} (G) \le \lambda^{(p)} (G^*) + \lambda^{(p)} (G\setminus G^*) $,  and the inequality (\ref{eqps}) implies $\lambda^{(p)} (G\setminus G^*) 
\le (2e(G\setminus G^*))^{1-1/p} = o(n^{2-2/p})$. 
Since $G^*$ is $K_{r+1}$-free, the Kang--Nikiforov result (\ref{eqKN}) 
implies $\lambda^{(p)} (G^*) \le \mathrm{ex}_{\lambda}^{(p)} (n,K_{r+1}) 
=\lambda^{(p)} (T_r(n))$. Thus, we get $\lambda^{(p)} (G) \le \lambda^{(p)} (T_r(n)) + o(n^{2-2/p})$. This completes the proof. 
\end{proof}

Our unified treatment of 
Theorems \ref{thm14} and \ref{thm16} 
stated in Subsection 3.1 
can allow us to generalize the spectral stability 
theorem in terms of  the $p$-spectral radius.

\begin{theorem}  \label{thm32}
Let $F$ be a graph with $\chi (F)=r+1\ge 3$. For every $p>1$ and 
$\varepsilon >0$,  
there exist $\delta >0$ and $n_0$ such that 
if $G$ is a graph on $n\ge n_0$ vertices, 
 and $G$ is $F$-free such that 
$\lambda^{(p)}(G)\ge (1- \frac{1}{r} - \delta) n^{2- (2/p)}$, 
 then  the edit distance $d(G,T_r(n)) \le \varepsilon n^2$. 
\end{theorem}

Theorem \ref{thm32} extends both Theorem \ref{thm12} and Theorem \ref{thm14} 
by applying (\ref{eqeqeq}). 

\begin{proof}
The proof is short and similar. It is based on applying 
Theorem \ref{thm12} and Lemma \ref{lemEFR}. 
There are two  differences in the proof. 
The first is that $ \lambda^{(p)} (G) \le 
\lambda^{(p)} (G^*) + \lambda^{(p)} (G\setminus G^*)$. 
The second is an extension of Lemma \ref{thmnikiforov}, which states that  
$\lambda^{(p)}(G) \le (2m)^{1-1/p}(1-\frac{1}{r})^{1/p}$ 
whenever $G$ is an $m$-edge $K_{r+1}$-free graph; see, e.g.,  \cite[Theorem 3]{KN2014}. 
\end{proof}

In the above, we established the stability theorem for $p$-spectral radius. 
 Under the similar line of our proof of Theorem \ref{thmcri}, 
 applying the $p$-spectral stability theorem can  allow us to   
 extend Theorem \ref{thmcri} and 
prove the exact $p$-spectral Tur\'{a}n function  
 for every color-critical graph and  real value $p>1$. 
Moreover, it is possible to extend the usual spectral extremal 
results to the $p$-spectral radius by applying 
the $p$-spectral stability  result. 

\begin{theorem}
If $F$ is a  graph with a critical edge and $\chi (F)=r+1$ where $r\ge 2$,  
then for every real number $p>1$, there exists an $n_0=n_0(F,p)$ 
such that  
\[    \mathrm{ex}_{\lambda}^{(p)}(n,F)=\lambda^{(p)}(T_r(n))  \]
 holds for all $n\ge n_0$, and 
the unique extremal graph is the Tur\'{a}n graph $T_r(n)$. 
\end{theorem}

\medskip 

\noindent 
{\bf Remark.}  
We remark that Kang and Nikiforov \cite[Theorem 6]{KN2014}, 
 Keevash, Lenz and Mubayi \cite[Corollary 1.5]{KLM2014} independently
proved  the same result with a different method.

\subsection{The minimum degree version} 

In this subsection, we shall consider 
the extremal graph problems in term of the minimum degree. 
Recall that $\delta (G)$ is the minimum degree of $G$. 
We define $\mathrm{ex}_{\delta}(n,F)$ 
to be the largest minimum degree 
in an $n$-vertex  graph that contains no copy of $F$, that is, 
\[ \mathrm{ex}_{\delta}(n,F):=\max \bigl\{ \delta(G): |G|=n~\text{and}~F\nsubseteq G \bigr\}. \]

First of all, we prove the degree version of Tur\'{a}n's theorem.\footnote{We are not sure whether 
Theorem \ref{degturan}  have already appeared in the literature, 
though we have not yet found this theorem elsewhere, 
so we  intend to present a proof here for completeness. } 

\begin{theorem}  \label{degturan}
If $G$ is an $n$-vertex graph containing no copy of $K_{r+1}$, 
then 
\[  \delta (G)\le \delta (T_r(n)). \]
Moreover, the equality holds if and only if $G=T_r(n)$.   
\end{theorem}

\medskip 

Before starting the proof, 
we show that 
the degree Tur\'{a}n theorem implies the classical Tur\'{a}n theorem. 
Indeed, given an $n$-vertex $K_{r+1}$-free graph $G$, 
the degree Tur\'{a}n Theorem \ref{degturan} implies 
$\delta (G)\le \delta (T_r(n))$. We delete a vertex of minimum degree, 
and the resulting graph $G'$ has 
$n-1 $ vertices with $e(G')=e(G)-\delta (G)$ edges. 
Note that $G'$ has no copy of $K_{r+1}$. By the inductive hypothesis, 
we obtain $e(G')\le e(T_r(n-1))$. 
Thus we have $e(G)=e(G') + \delta (G)\le e(T_r(n-1)) + \delta (T_r(n)) =e(T_r(n))$.  
Moreover, the equality holds if and only if 
$e(G')=e(T_r(n-1))$ and $\delta (G)=\delta (T_r(n))$. Hence 
the equality case of the degree Tur\'{a}n theorem implies $G=T_r(n)$. 

\begin{proof}
Note that $\delta (T_r(n))= n- \lceil \frac{n}{r} \rceil$ and 
$\delta (G) \ge n- \lceil \frac{n}{r} \rceil +1$.  
Suppose that $n=qr+s$, where $1\le s \le r$. 
Hence $\delta (G) \ge n- (q+1) +1=n-q$. 
Let $u_1,\ldots ,u_r$ be $r$ distinct vertices in $G$. 
Then we have  $|\cap_{i=1}^r N(u_i)| 
\ge \sum_{i=1}^r |N(u_i)| - (r-1)|\cup_{i=1}^r N(u_i)| 
\ge r(n-q) - (r-1) n=s\ge 1$, which implies that 
$G$ contains a copy of $K_{r+1}$. 
\end{proof}

In what follows,  
we  give another way to show the degree version. 
More precisely, we shall show that 
  the degree version for such extremal problem 
can be deduced from the classical edge version. 

\begin{proof}[Second proof] 
Let $G$ be a $K_{r+1}$-free graph on $n$ vertices. 
By the Tur\'{a}n theorem, we know that $e(G)\le e(T_r(n))$. 
We assume on the contrary that $ \delta (G) \ge \delta (T_r(n)) +1$. 
We assume that $s$ vertices of $T_r(n)$ have
 degree $\delta (T_r(n))$, 
and $t$ vertices  have degree 
$\delta (T_r(n))+1$, where $s+t=n$ and $0\le t<n$. 
We have 
$2e(G)=\sum d_i \ge n\delta (T_r(n)) +n >n\delta (T_r(n)) + t =
 2e(T_r(n))$,  a contradiction. 
Therefore, we  get $\delta (G) \le \delta (T_r(n))$. 
Moreover, equality holds if and only if $e(G) = e(T_r(n))$, 
and then $G=T_r(n)$. 
So the Tur\'{a}n theorem implies the degree version. 
On the other hand, the degree version can deduce the Tur\'{a}n theorem  by deleting 
a vertex with  minimum degree, we have 
$e(G)=e(G-v) + d(v)\le e(T_r(n-1)) + \delta (T_r(n))= e(T_r(n))$. 
\end{proof}

As mentioned before the proof, 
we know that the degree version implies the classical Tur\'{a}n theorem. 
The proof of Theorem \ref{degturan} seems not complicated. However, 
it is surprising that Theorem \ref{degturan} 
seems not well-known in extremal graph community. 
Moreover, the second proof reveals an interesting phenomenon 
that both the degree version 
 and the edge version of the Tur\'{a}n theorem are equivalent. 
 
 \medskip 
 \noindent 
{\bf Remark.}~ 
For some extremal graph problems, 
if the extremal graph 
is regular or nearly regular, then almost all degree version 
can imply the usual edge version. 
For example, the extremal hypergraph Tur\'{a}n problem for 
the Fano plane. Generally speaking, for some problems with 
the extremal graphs far from being regular, 
these two versions can  not be converted to each other. 
For instance, the Erd\H{o}s--Ko--Rado theorem 
and its degree version, the Hilton--Milner theorem and its degree version.

\medskip 

For completeness, 
we next are going to present the degree versions of the Erd\H{o}s--Stone--Simonovits 
theorem and Erd\H{o}s--Simonovits stability theorem. 
The proofs can be given as a direct consequence of Theorems \ref{thm11} and \ref{thm12}. 

\begin{theorem}
If $F$ is a graph with chromatic number $\chi (F)=r+1$, 
then 
\[  \mathrm{ex}_{\delta} (n,F) =\delta (T_r(n)) + o(n)=
\left( 1-\frac{1}{r} + o(1)\right) n . \] 
\end{theorem}

\begin{theorem}[Degree stability theorem] 
Let $F$ be a graph with $\chi (F)=r+1\ge 3$. For every $\varepsilon >0$,  
there exist $\delta >0$ and $n_0$ such that 
if $G$ is a graph on $n\ge n_0$ vertices, 
 and $G$ is $F$-free such that 
$\delta (G)\ge (1- \frac{1}{r} - \delta) n$, 
 then  the edit distance $d(G,T_r(n)) \le \varepsilon n^2$. 
\end{theorem}

Applying the degree stability theorem 
and the techniques of the proof in Subsection \ref{Sec3.2}, 
we can similarly prove the following corresponding theorem 
for all color-critical graphs.  

\begin{theorem} 
If $F$ is a  graph with a critical edge and $\chi (F)=r+1$ where $r\ge 2$,  
then there exists an $n_0=n_0(F)$ 
such that  
\[    \mathrm{ex}_{\delta}(n,F)= \delta(T_r(n))  \]
 holds for all $n\ge n_0$, and 
the unique extremal graph is the Tur\'{a}n graph $T_r(n)$. 
\end{theorem}

\subsection{The signless Laplacian spectral radius}

Given a graph $G$, 
 the signless Laplacian matrix of $G$ is defined as 
 $Q(G)=D(G) +A(G)$, 
 where $D(G)=\mathrm{diag}(d_1,\ldots ,d_n)$ is the degree diagonal matrix and 
 $A(G)$ is the adjacency matrix. We denote by $q(G)$ the largest eigenvalue of 
 $Q(G)$. Since $Q(G)$ is a positive semidefinite matrix, 
its largest eigenvalue is actually the spectral radius. 
Hence we call $q(G)$ the signless Laplacian spectral radius of $G$.

A natural question is to extend the above-mentioned
 results on the adjacency spectral radius 
 to that of the signless Laplacian spectral radius. 
We define $\mathrm{ex}_{q}(n,F)$ 
to be the largest eigenvalue of the signless Laplacian   matrix 
in an $n$-vertex  graph that contains no copy of $F$. That is, 
\[ \mathrm{ex}_{q}(n,F):=\max \bigl\{ q(G): |G|=n~\text{and}~F\nsubseteq G \bigr\}. \]
Note that $Q(G)=D(G)-A(G) + 2A(G)$ and $D(G)-A(G)$ 
is positive semidefinite. It is known by the Weyl theorem 
for monotonicity of eigenvalues that  
$ 2\lambda (G)\le q(G)$. Thus any upper bound on $q(G)$ yields an upper bound on $\lambda(G)$.

In 2013, He, Jin and Zhang \cite{HJZ2013} proved that 
if $G$ is an $n$-vertex $K_{r+1}$-free graph, 
then $q(G)\le q(T_r(n))$. 
Moreover, the equality holds if and only if 
$G$ is a complete bipartite graph (not necessarily balanced) 
for $r=2$ or $G=T_r(n)$ for $r\ge 3$.  
In other words, we have 
\begin{equation} \label{eqTuransign}
\mathrm{ex}_{q}(n,K_{r+1}) =q(T_r(n)). 
\end{equation}
It is worth noting that the extremal graphs for the case $r=2$ 
are not unique. This phenomenon is surprisingly different from the 
extremal problem on the adjacency  spectral radius.  
Moreover, the signless Laplacian spectral Tur\'{a}n theorem (\ref{eqTuransign}) 
also implies the 
classical Tur\'{a}n theorem (\ref{eqTuran}); see \cite[Corollary 2.5]{HJZ2013}.

It is natural to consider the extremal problem for signless Laplacian radius for general graphs.  
However, the Erd\H{o}s--Stone--Simonovits type result 
and 
the Erd\H{o}s--Simonovits type stability result 
in terms of the signless Laplacian spectral radius do not hold.  

\medskip 
\noindent 
{\bf Remark.} 
{\it Let $F$ be a graph with chromatic number $\chi (F)=r+1$. 
The Erd\H{o}s--Stone type result $  \mathrm{ex}_{q} (n,F) = 
\left( 1-\frac{1}{r} \right) 2n + o(n)$ 
is not necessary to be true.  } 

\medskip 
The result is negative in the case $F=C_{2k+2}$ for every integer  $k\ge 1$. 

\begin{itemize}

\item 
First of all,  we take $F=C_4$ as a counter-example.  
When $n$ is odd, let $F_n$ 
be the friendship graph of order $n$, that is, 
$F_n=K_1\vee \frac{n-1}{2}K_2$;  
When $n$ is even, let $F_n$ be the graph obtained from 
$F_{n-1}$ by hanging an extra 
edge to its center. 
In other words, the $F_n$ can be viewed as 
a graph obtained from $K_{1,n-1}$ by adding a maximum 
matching within the independent set. Note  that 
$F_n$ is $C_4$-free. Upon computation, 
we get $q(F_n)=\frac{n+2 + \sqrt{n^2-4n +12}}{2}$ for odd $n$; and $q(F_n)= \frac{ n+1 + \sqrt{n^2 -2n+9}}{2}$ for even $n$. 
Thus we have 
$\mathrm{ex}_q(n, C_4) \ge  n + o(1)$. 
But $\chi (C_4)=2$ and $(1-\frac{1}{r})2n + o(n)=o(n)$. 

\item 
For the case  $k\ge 2$, 
let $S_{n,k}$ be the graph consisting of a clique on $k$ vertices and an independent set on $n-k$ vertices in which each vertex of the clique is adjacent to each vertex of the independent set. 
We can observe that $S_{n,k}$ does not contain $C_{2k+2}$ as a subgraph. Furthermore, 
let $S_{n,k}^+$ be the graph obtained from $S_{n,k}$ 
by adding an edge to the independent set $I_{n-k}$. 
In the language of join of graphs, we have  
$S_{n,k}^+=K_k \vee I_{n-k}^+$.  
Clearly, we can see that 
$S_{n,k}^+$  is still $C_{2k+2}$-free and 
\[  q(S_{n,k}^+) >  q(S_{n,k}) = \frac{n+2k-2+ \sqrt{(n+2k-2)^2 - 8k^2 +8k
}}{2}.   \]
Hence either the graph $S_{n,k}$ or $S_{n,k}^+$ can yield $\mathrm{ex}_q(n,C_{2k+2}) \ge n+ o(n)$. 
However we have $\chi (C_{2k+2})=2$ and $(1-\frac{1}{r})2n + o(n)=o(n)$.  \qedhere 
\end{itemize}

It is worth noting that 
Freitas, Nikiforov and Patuzzi \cite{FNP2013} showed that 
$\mathrm{ex}(n,C_4)=q(F_n)$ and 
 $F_n$ is the unique extremal graph.  
Moreover, Nikiforov and Yuan \cite{NY2015}  proved 
$\mathrm{ex}_q(n,C_{2k+2}) = q(S_{n,k}^+)$ 
for all $k\ge 2, n\ge 400 k^2$ 
and $S_{n,k}^+$ is the unique extremal graph.  
It is a problem whether the condition can be relax to $n\ge ck$ for some constant $c>0$. 
In addition, it is meaningful to determine graphs $F$  satisfying
 $\mathrm{ex}_{q} (n,F) = 
\left( 1-\frac{1}{r} \right) 2n + o(n)$. 

\medskip 

For the signless Laplacian radius, one may 
make the following remark.

\medskip 
\noindent 
{\bf Remark.} 
{\it The following statement is not true: 
For any graph $F$ with $\chi (F)=r+1$, $r\ge 2$ 
and $\varepsilon >0$, 
there exist $\delta >0$ and $n_0$ such that 
if $n\ge n_0$ and $G$ is an $F$-free graph 
on $n$ vertices with $q(G)\ge (1-\frac{1}{r} -\delta) 2n$, 
then the edit distance $d(G,T_r(n)) \le \varepsilon n^2$. }

\medskip 
Indeed, we now consider the case $F=C_{2k+1}$ 
or $F=F_{2k+1}$ for every $k\ge 2$, where $F_{2k+1}$ is defined as 
the  graph consisting of $k$ triangles 
intersecting in exactly one common vertex. 
Note that $\chi (C_{2k+1})=\chi (F_{2k+1})=3$. 
If $G$ is $C_{2k+1}$-free (or $F_{2k+1}$-free) 
with $q(G)\ge (\frac{1}{2} -o(1)) 2n$, 
then we can not get $d(G,T_2(n))=o(n^2)$. 
The reasons are stated as below. 
Recall that  $S_{n,k}$ is the graph consisting of a clique on $k$ vertices and an independent set on $n-k$ vertices in which each vertex of the clique is adjacent to each vertex of the independent set. 
Clearly, we can see that 
$S_{n,k}$  does not contain $C_{2k+1}$ and $F_{2k+1}$ as a subgraph. 
Taking $G=S_{n,k}$, we  calculate that 
$q(S_{n,k})\sim n+2k-2 
\ge (\frac{1}{2} - o(1))2n$. 
However, the fact $e(S_{n,k})={k \choose 2} + k(n-k)$, 
which together with $T_2(n)= \lfloor n^2/4 \rfloor$ 
implies $d(G,T_2(n)) \ge \Omega (n^2)$ 
for fixed integer $k$.

\medskip 

We mention here  that  for $k=1$, 
the $C_3$-free graphs attaining the maximum   signless Laplacian radius are complete bipartite graphs \cite{HJZ2013}. 
For $k\ge 2$, the $C_{2k+1}$-free graph attains the maximum 
 signless Laplacian radius is uniquely the split graph $S_{n,k}
 =K_k \vee I_{n-k}$.  This result was proved by 
 Freitas, Nikiforov and Patuzzi \cite{FNP2013} 
 for $k=2$ and $n\ge 6$, 
 and by Yuan \cite{Yuan2014}  for $k\ge 3$ and $n\ge 110k^2$. 
In addition, 
the $T_k$-free graph attains the maximum 
 signless Laplacian radius is also the split graph $S_{n,k}$. 
 This result was recently proved  by  Zhao, Huang and 
 Guo \cite{ZHG21} for $k\ge 2$ and $n\ge 3k^2-k-2$. 
It is interesting that whether these results are valid for $n\ge ck$ for some $c>0$. 

\medskip  
\noindent 
{\bf Remark.} {\it The above example for $F=C_{2k+1} $ implies that  the following statement is not true. 
If $F$ is a  graph with a critical edge and $\chi (F)=r+1$ where $r\ge 2$,  
then there exists an $n_0=n_0(F)$ 
such that   $  \mathrm{ex}_{q}(n,F)= q(T_r(n)) $ 
 holds for all $n\ge n_0$, and 
the unique extremal graph is the Tur\'{a}n graph $T_r(n)$. }

\section{Concluding remarks}

We remark that the statement of Theorem \ref{thm14} and Theorem \ref{thmcri} 
 are  simplifications of \cite[Theorem 2]{Niki2009jgt} and \cite[Theorem 2]{Niki2009ejc}, respectively. 
 Nikiforov's original stability theorems (spectral and non-spectral) are a lot more detailed 
 and contain a lot more quantitative bound. 
The purpose was to present theorems that can be used as tools in other theorems. 
It seems impossible that his original stability theorems can be deduced from the Erd\H{o}s--Simonovits theorem, or by the Regularity Lemma.

Recall that $A(G)$ and $D(G)$ are the adjacency matrix and 
degree diagonal matrix of $G$. The signless Laplacian matrix of 
$G$ is defined as $Q(G)=D(G) + A(G)$. 
It was proposed by Nikiforov  \cite{NikiMerge} 
to study the family of matrices 
$A_{\alpha}$ defined for any real $\alpha \in [0,1]$ as 
\[  A_{\alpha}(G) =\alpha D(G) + (1-\alpha )A(G). \]
In particular, we can see that $A_0(G)=A(G)$ and $2A_{1/2}(G)=Q(G)$. 
Nikiforov told us that 
he \cite[Theorem 27]{NikiMerge} proved a unified extension of both (\ref{eqeq2}) and (\ref{eqTuransign}), which states that 
for every $r\ge 2$ and every $K_{r+1}$-free graph $G$, 
if $0\le \alpha <1-\frac{1}{r}$, then $\lambda (A_{\alpha}(G)) 
< \lambda (A_{\alpha}(T_r(n)))$, unless $G=T_r(n)$; 
if $\alpha =1-\frac{1}{r}$, then $\lambda (A_{\alpha}(G)) < 
(1-\frac{1}{r})n$, unless $G$ is a complete $r$-partite graph; 
if $1-\frac{1}{r}< \lambda <1$, then $\lambda (A_{\alpha}(G)) 
< \lambda (A_{\alpha}(S_{n,r-1}))$, unless $G=S_{n,r-1}$, 
where $S_{n,k}=K_k \vee I_{n-k}$. 
From this evidence, it is possible to extend the results of our  paper 
into the $A_{\alpha}$-spectral radius in the range $ \alpha \in [0,1-\frac{1}{r})$.

\medskip 

At the end of this paper, we propose a problem related to the recent progress on stability type theorems. 
 Let $f_r(n,t)$ be the smallest number such that 
any $K_{r+1}$-free graph $G$ with $e(G)\ge e(T_r(n))-t$ 
edges can be made $r$-partite by deleting at most $f_r(n,t)$ edges. 
The Erd\H{o}s--Simonovits stability theorem implies that $f_r(n,t)=o(n^2)$ if $t=o(n^2)$. 
 Very recently, Kor\'{a}ndi, Roberts and 
Scott \cite{KRS2021} proved that $f_r(n,t)$ is witnessed by a pentagonal Tur\'{a}n graph 
if $t$ is small enough, which confirmed a conjecture of Balogh, Clristian, Lavrov, Lidick\'{y} and Pfender \cite{BCLLP2021}. 

\begin{theorem}[Kor\'{a}ndi--Roberts--Scott, 2021]
For every $r\ge 2$, there is a $\delta_r >0$ such that 
if $G$ is a $K_{r+1}$-free graph on $n$ vertices with 
$e(G)\ge e(T_r(n)) -\delta_r n^2$ edges, then there is a pentagonal 
Tur\'{a}n graph $G^*$ on $n$ vertices with $e(G^*)\ge e(G)$ 
and  $D_r(G^*) \ge D_r(G)$. 
\end{theorem}

It is natural to consider the corresponding spectral problem. 

\begin{problem} 
For every $r\ge 2$, there is a $\delta_r >0$ such that 
if $G$ is a $K_{r+1}$-free graph on $n$ vertices with 
$\lambda (G)\ge \lambda (T_r(n)) -\delta_r n$, then there is a pentagonal 
Tur\'{a}n graph $G^*$ on $n$ vertices with $\lambda(G^*)\ge \lambda(G)$ 
and  $D_r(G^*) \ge D_r(G)$. 
\end{problem}

\subsection*{Acknowledgements}
This paper is dedicated to Vladimir Nikiforov 
whose beautiful works on spectral graph theory inspire the author.  
The first author would like to thank Prof. Lihua Feng,   
who introduced and encouraged him to the study of 
fascinating spectral graph theory
when he was a graduate student at  Central South University.  
Thanks also go to Prof. Vladimir Nikiforov 
for valuable comments and for pointing out references \cite{Zykov1949,Kha1977}.


\frenchspacing
\begin{thebibliography}{99}


\bibitem{AS2016} 
N. Alon, C. Shikhelman, 
Many $T$ copies in $H$-free graphs, 
J. Combin. Theory Ser. B 121 (2016) 146--172. 

\bibitem{BCLLP2021}
J. Balogh,  F.C. Clemen,  M. Lavrov, B. Lidick\'{y}, F. Pfender,  
Making $K_{r+1}$-free graphs $r$-partite, 
Combin. Probab. Comput. 30 (4) (2021) 609--618.

\bibitem{BN2007jctb} 
B. Bollob\'{a}s, V. Nikiforov, 
Cliques and the spectral radius, 
J. Combin. Theory Ser. B 97 (2007) 859--865.  

\bibitem{BN2008dm} 
B. Bollob\'{a}s, V. Nikiforov, 
Joints in graphs, Discrete Math. 308 (2008) 9--19. 


 \bibitem{CFTZ20}
 S. Cioab\u{a}, L.H. Feng, M. Tait, X.-D. Zhang, 
 The spectral radius of graphs with no intersecting triangles, 
 Electron. J. Combin. 27 (4) (2020) P4.22. 
 
 \bibitem{CDT2022}
 S. Cioab\u{a}, D.N. Desai, M. Tait, 
 The spectral radius of graphs with no odd wheels, 
 European J. Combin. 99 (2022) 103420. 

\bibitem{FNP2013} 
M.A.A. de Freitas, V. Nikiforov, L. Patuzzi, 
Maxima of the $Q$-index: forbidden 4-cycle and 5-cycle, 
Electron. J. Linear Algebra 26 (2013) 905--916. 

\bibitem{DKLNTW2021}
D.N. Desai, L. Kang, Y. Li, Z. Ni, M. Tait, J. Wang, 
Spectral extremal graphs for intersecting cliques, 
 18 pages, (2021), arXiv: 2108.03587v2.  See \url{https://arxiv.org/abs/2108.03587v2} 

\bibitem{ES66}
P. Erd\H{o}s, M. Simonovits, 
A limit theorem in graph theory, Stud. Sci. Math. Hungar. 1 (1966) 51--57.  

\bibitem{ES46}
P. Erd\H{o}s, A.H. Stone, 
On the structure of linear graphs, 
Bull. Amer. Math. Soc. 52 (1946) 1087--1091. 

\bibitem{Erd1962}
P. Erd\H{o}s, 
On the number of complete subgraphs contained in certain graphs, 
Magy. Tud. Akad. Mat. Kut. Int\'{e}z. K\"{o}zl. 7 (1962) 459--474. 


\bibitem{Erd1966Sta1}
P. Erd\H{o}s, Some recent results on extremal problems in graph theory (Results), In: Theory of Graphs (International Symposium Rome, 1966), Gordon and Breach, New York, Dunod, Paris, 1966, pp. 117--123.

\bibitem{Erd1966Sta2} 
P. Erd\H{o}s, On some new inequalities concerning extremal properties of graphs, In: Theory of Graphs (Proceedings of the Colloquium, Tihany, 1966), Academic Press, New York, 1968, pp. 77--81.  

\bibitem{EFR1986}
P. Erd\H{o}s,  P. Frankl, V. R\"{o}dl, 
The asymptotic number of graphs not containing a fixed subgraph 
and a problem for hypergraphs having no exponent, 
Graphs Combin. 2 (2) (1986) 113--121. 



\bibitem{FS13} 
Z. F\"uredi,  M. Simonovits,
The history of degenerate (bipartite) extremal graph problems,  
in Erd\H{o}s Centennial, 
Bolyai Soc. Math. Stud., 25, 
J\'{a}nos Bolyai Math. Soc., Budapest, 2013, 
pp. 169--264. 

\bibitem{Fur2015} 
Z. F\"{u}redi,  
A proof of the stability of extremal graphs,  
Simonovits' stability from Szemer\'{e}di's regularity, 
J. Combin. Theory Ser. B 115 (2015) 66--71.  


\bibitem{Gui1996}
B.D. Guiduli, 
 Spectral extrema for graphs, 
 Ph. D. Thesis, 
University of Chicago, December 1996. 
See \url{http://people.cs.uchicago.edu/~laci/students/guiduli-phd.pdf}

\bibitem{HJZ2013}
B. He, Y.-L. Jin, X.-D. Zhang,  
Sharp bounds for the signless Laplacian spectral radius in terms of clique  number, Linear Algebra Appl. 438 (2013) 3851--3861.


\bibitem{KN2014} 
L. Kang, V. Nikiforov, 
Extremal problem for the $p$-spectral radius of graphs, 
Electronic J. Combin. 21 (3) (2014) 87--101. 



\bibitem{Keevash11}
P. Keevash, 
Hypergraph Tur\'{a}n problems, in Surveys in Combinatorics, 
Cambridge University Press, Cambridge, 2011, pp. 83--140. 


\bibitem{KLM2014}
P. Keevash, J. Lenz, D. Mubayi, 
Spectral extremal problems for hypergraphs, 
SIAM J. Discrete Math. 28 (4) (2014) 1838--1854.

\bibitem{Kha1977}
N. Khadzhiivanov, Inequalities for graphs (in Russian), C. R. Acad. Sci. Bul. 30 
(1977) 793--796.


\bibitem{KRS2021}
D. Kor\'{a}ndi, A. Roberts, A. Scott, 
Exact stability for Tur\'{a}n theorem, 
Advances in Combinatorics (9) 2021, 17pp. 
See \url{https://doi.org/10.19086/aic.31079}. 

\bibitem{LP2021}
Y. Li, Y. Peng, 
The spectral radius of graphs with no intersecting odd cycles, 
22 pages, Discrete Math. (2022), to appear, 
arXiv: 2106.00587. See \url{https://arxiv.org/abs/2106.00587}. 
 

\bibitem{Liu2021}
X. Liu, 
New short proofs to some stability theorems, 
European J. Combin.  96 (2021)  103350. 


\bibitem{MY2018}
J. Ma, X. Yu, Some notes on Extremal Combinatorics, 
a lecture notes for 
Tianyuan Mathematics Foundation
 2018 Summer School on Graph Theory. 


\bibitem{MQ2020}
J. Ma, Y. Qiu, Some sharp results on the generalized Tur\'{a}n numbers, 
European J. Combin. 84 (2020) 103026. 



\bibitem{Niki2002cpc} 
V. Nikiforov, 
Some inequalities for the largest eigenvalue of a graph, 
Combin. Probab. Comput. 11 (2002) 179--189.  



\bibitem{Niki2007laa2} 
V. Nikiforov, Bounds on graph eigenvalues II, 
Linear Algebra Appl. 427 (2007) 183--189. 


\bibitem{Niki2008blms} 
V. Nikiforov, 
Graphs with many $r$-cliques have large complete $r$-partite subgraphs, 
Bull. London Math. Soc. 40 (2008) 23--25. 

\bibitem{Niki2008} 
V. Nikiforov, A spectral condition for odd cycles in graphs, 
Linear Algebra Appl. 428 (2008) 1492--1498. 

\bibitem{Niki2009jctb} 
V. Nikiforov, More spectral bounds on the clique and independence numbers, 
J. Combin. Theory Ser. B 99 (6) (2009) 819--826. 


\bibitem{Niki2009cpc} 
V. Nikiforov, 
A spectral Erd\H{o}s--Stone--Bollob\'{a}s theorem, 
Combin. Probab.  Comput. 18 (2009) 455--458.  


\bibitem{Niki2009ejc} 
V. Nikiforov, 
Spectral saturation: inverting the spectral Tur\'{a}n theorem,  
Electron. J. Combin. 16 (1) (2009) R 33.  

\bibitem{Niki2009jgt} 
V. Nikiforov, 
Stability for large forbidden subgraphs,  
J. Graph Theory 62 (4) (2009) 362--368.   

\bibitem{Niki2010dm} 
V. Nikiforov, 
Tur\'{a}n's theorem inverted,  
Discrete Math. 310 (1) (2010) 125--131. 


\bibitem{NikifSurvey}
V. Nikiforov,  Some new results in extremal graph theory, 
 Surveys in Combinatorics,  London Math. Soc. Lecture Note Ser., 392, Cambridge Univ. Press, Cambridge, 2011, pp. 141--181. 
 
 
 \bibitem{Niki2013}
  V. Nikiforov,  
 An extension of Maclaurin’s inequality, 7 pages, 
 Preprint, 2006/2013, arXiv: math/0608199v3. 
 See \url{https://arxiv.org/abs/math/0608199v3}. 
 
 \bibitem{Niki2013}
 V. Nikiforov, 
Analytic methods for uniform hypergraphs, 
Linear Algebra Appl. 457 (2014) 455--535.
 
 \bibitem{NY2015} 
V. Nikiforov, X. Yuan, 
Maxima of the $Q$-index: Forbidden even cycles, 
Linear Algebra Appl. 471 (2015) 636--653.  

\bibitem{NikiMerge}
 V. Nikiforov, 
Merging the $A$- and $Q$-spectral theories, 
Appl. Anal. Discrete Math. 11 (1) (2017) 81--107. 

\bibitem{Nikpri2021}
V. Nikiforov, Private communication, December 28, 2021. 

\bibitem{PSS2018}
K. Popielarz, J. Sahasrabudhe, R. Snyder,  
A stability theorem for maximal $K_{r+1}$-free graphs, 
J. Combin. Theory Ser. B 132 (2018) 236--257.

\bibitem{RS2006}
V. R\"{o}dl,  J. Skokan, 
Applications of the regularity lemma for uniform hypergraphs,  
Random Structures Algorithms 28 (2) (2006) 180--194. 

 \bibitem{Sim1966} 
M. Simonovits, 
A method for solving extremal problems in graph theory, 
stability problems, in: Theory of Graphs, 
Proc. Colloq., Tihany, 1966, Academic Press, New York, (1968), pp. 279--319. 

\bibitem{Sim13}
M. Simonovits,  
Paul Erd\H{o}s' influence on Extremal graph theory, 
in The Mathematics of Paul Erd\H{o}s II, 
R.L. Graham, Springer, New York, 2013, pp. 245--311. 


\bibitem{SS1982}
V. S\'{o}s, E. Straus, Extremals of functions on graphs with applications to graphs and 
hypergraphs, J. Combin. Theory Ser. B 32 (1982)  246--257.

\bibitem{Turan41}
P. Tur\'{a}n, 
On an extremal problem in graph theory, 
Mat. Fiz. Lapok 48 (1941), pp. 436--452. 
(in Hungarian). 

\bibitem{Yuan2014} 
X. Yuan, 
Maxima of the $Q$-index: Forbidden odd cycles, 
Linear Algebra Appl. 458 (2014) 207--216. 


\bibitem{ZHG21}
Y. Zhao, X.Y. Huang, H. Guo, 
The signless Laplacian spectral radius of graphs 
with no intersecting triangles, 
Linear Algebra Appl. 618 (2021) 12--21. 

\bibitem{Zykov1949}
A.A. Zykov, On some properties of linear complexes (in Russian), 
Mat. Sbornik N.S. 24 (66) (1949) 163--188. 

\end{thebibliography}
\end{document}